\documentclass[12pt]{amsart}

\usepackage{latexsym, amssymb}

\newcommand{\be}{\begin{equation}}
\newcommand{\ee}{\end{equation}}
\newcommand{\beq}{\begin{eqnarray}}
\newcommand{\eeq}{\end{eqnarray}}

\def\R{{\mathfrak R}}

\newtheorem{prop}{Proposition}[section]
\newtheorem{theo}[prop]{Theorem}
\newtheorem{lemm}[prop]{Lemma}
\newtheorem{coro}[prop]{Corollary}

\newtheorem{defi}[prop]{Definition}

\def\begeq{\begin{equation}}
\def\endeq{\end{equation}}
\def\p{\partial}

\def\R{\Bbb R}

\begin{document}

\title {The  isoperimetric inequality  on asymptotically flat manifolds with nonnegative scalar curvature}

\begin{abstract}
In this note, we consider the isoperimetric inequality on an asymptotically flat manifold with nonnegative scalar curvature, and improve it by using  Hawking mass. We also obtain a rigidity result when equality holds for the classical isoperimetric inequality on an asymptotically flat manifold with nonnegative scalar curvature.
\end{abstract}


\renewcommand{\subjclassname}{\textup{2000} Mathematics Subject Classification}

\author{Yuguang Shi$^\dag$}

\address{Yuguang Shi, Key Laboratory of Pure and Applied mathematics, School of Mathematical Sciences, Peking University,
Beijing, 100871, P.R. China.} \email{ygshi@math.pku.edu.cn}

\thanks{$^\dag$ Research partially supported by   NSF grant of China 10725101 and 10990013.}

\date{2015}
\maketitle

\markboth {Yuguang Shi}{}

\section {introduction}

   The isoperimetric inequality  and isoperimetric surfaces have a very long history and many important applications in mathematics, see e.g. \cite{B}, \cite{CY}.  Huisken has observed that ADM mass of an  asymptotically flat manifold (see Definition \ref{defaf} below) appears in the expansion of isoperimetric ratio when the volume is large enough, see \cite{Hu} and \cite{EM2} (for the case of  coordinates sphere, see \cite{FST}). Inspired by these facts, it is natural to ask if there is any relationship between the isoperimetric inequality and quasi-local mass for any fixed  enclosed volume. In this short note,  we are able to use the Hawking mass to improve the isoperimetric inequality in some cases. In order to present our result, we need some notions.

\begin{defi}\label{defaf}
A complete and connected  three-manifold $(M^3,g)$ is said to be {\rm asymptotically flat}
(AF) (with one end) if there are a  positive constant $C>0$ and a compact subset $K$
such that $M\setminus K$ is diffeomorphic to $\R^3\setminus B_R(0)$
for some $R>0$ and in the standard coordinates in $\R^3$, the metric
$g$ satisfies:
\begin{equation} \label{daf1}
g_{ij}=\delta_{ij}+\sigma_{ij}
\end{equation}
with
\begin{equation} \label{daf2}
|\sigma_{ij}|+r|\p \sigma_{ij}|+r^2|\p\p\sigma_{ij}|\leq C r^{-1},
\end{equation}
 where $r$ and $\p$ denote the Euclidean distance and standard derivative operator on $\R^3$
respectively. The region $M\setminus K$ is called the end of $M$.
\end{defi}

An original idea of \cite{BC} is to use  the  weak solution of inverse mean curvature (\ref{imcf}) to estimate the volumes  of isoperimetric regions in an asymptotically hyperbolic manifold. Inspired by this, we use the same idea to investigate the same problem in the case of  AF manifolds.  More specifically,  for any $x\in M$, it is proved here that there is a weak solution $(G_t)_{t>-\infty}$ of (\ref{imcf}) with initial condition $\{x\}$ in \cite{HI}. One important property for this weak solution  is that for each $t\in \R$, $(G_t)$ has the least boundary area among all domains containing it, i.e. $(G_t)$ is a minimizing hull in $(M^3, g)$. Another interesting property is that  the Hawking  mass of $K_t =\p G_t$  which is defined as
$$
 m_H (t)=\frac{(Area (K_t))^\frac12}{(16\pi)^\frac32}(16\pi - \int_{K_t}H^2),
$$
  is nondecreasing in $t$; here, $H$ is the mean curvature of $K_t(x) =\p G_t$ with respect to outward unit normal vector. By using this quantity, we are able to estimate the  area of $K_t$  in terms of the volume of $G_t$, see (\ref{isoperineq1}) below; hence, we obtain Theorem \ref{comparisontheorem1}.
To do that, we need to parametrize $t$ by $v$, which is the volume of $G_t$,  and it turns out that this function $t(v)$ is Lipschitz; for details, see  Lemma \ref{functiontofv} below.  Let $m(v)= m_H (t(v))$, $B(v)= Area (K_{t(v)})$, and

\begin{equation}
\begin{split}
A(v)= \inf\{&\mathcal{H}^2(\partial^* \Omega):  \Omega\subset M \text{ is a Borel set with finite perimeter, and }\\
&\mathcal{L}^3 (\Omega)=v\}.
\end{split}\nonumber
\end{equation}
here, $\mathcal{H}^2$ is $2$-dimensional  Hausdorff measure for the reduced boundary of $\Omega$, and $\mathcal{L}^3 (\Omega)$  is the Lebesgue measure of $\Omega$ with respect to metric $g$.

Then our main result can be stated as follows

\begin{theo}\label{comparisontheorem1}
Suppose $(M^3, g)$ is an asymptotically flat (AF) manifold with nonnegative scalar curvature. Fix a point $o\in M$, for every  $v>0$, there is a $\rho>0$ so that for all $x\in M\setminus B_{\rho}(0)$ we have that

\begin{equation}\label{inequality1}
A(v)\leq (36\pi)^\frac13 \left(\int^v_0(1-(16\pi)^{\frac12}B^{-\frac12} (t) m(t))^\frac12  dt\right)^\frac23 .
\end{equation}
Where $m(v)$ is defined as above.

\end{theo}

When scalar curvature of $M$ is  non-negative, and $M$ satisfies some topological conditions,  we have that $m(v)\geq 0$. We see that in this case $A(v)\leq (36\pi)^\frac13 v^\frac23$.  Comparing this with the Euclidean case in which $m(v)=0$,  we observe the following heuristic phenomenon: {\it to enclose the same volume,  isoperimetric surfaces in a manifold with bigger mass have smaller area.} We believe such a phenomenon can also be observed in  the case of asymptotically hyperbolic  manifolds, and we will discuss this problem in a future paper. With these facts in mind, it is natural to ask what happens if there is  a $v_0>0$ with $A(v_0)=(36\pi)^\frac13 v_0 ^\frac23$?  Our next theorem gives an answer to this question.

\begin{theo}\label{rigidity}
Suppose $(M^3, g)$ is an asymptotically flat  manifold with nonnegative scalar curvature. Then there is a $v_0 >0$ with
$$
A(v_0)=(36\pi)^\frac13 v_0 ^\frac23
$$
  if and only if $(M^3, g)$ is isometric to $\R^3$.
\end{theo}

Inequality (\ref{isoperineq1}) below is crucial in the proof of Theorem \ref{comparisontheorem1} and Theorem \ref{rigidity}, and its equivalent version was first proved in \cite{BC} (see Proposition 3 in \cite{BC}), and the arguments here are quite similar to those in \cite{BC}.

Theorem \ref{comparisontheorem1} and Theorem \ref{rigidity} play important roles in the proof of existence of isoperimetric regions in a non-flat AF manifold. There are many results that focus on large isoperimetric regions in  an AF manifold, where the asymptotic regime plays an important role, see \cite{EM1} and references therein. However, there are very few results on  the existence of isoperimetric regions with medium size in an AF manifold. One difficulty is that the minimizing sequence of isoperimetric regions may drift off to infinity, while  Theorem \ref{comparisontheorem1} allows for control over these minimizing sequence in a certain sense and we may obtain the existence isoperimetric regions with any given volume. This was observed in the  very recent paper \cite{CCE}.

The  outline of the paper is as follows. In Section 2, we introduce some notions and basic facts of weak solutions of inverse mean curvature flow from \cite{HI}; in Section 3, we prove the main results.

{\bf Acknowledgements} The author is  grateful to Prof. Frank Morgan, Dr. Gang Li and Dr. Chao Bao for pointing out some typos and errors in the first version of the paper, and also would like to thank referees for many useful comments and suggestions which make the paper  clearer.
Especially, the author would like to thank one of referees for pointing out that the assumptions of Theorem \ref{comparisontheorem1} and Theorem \ref{rigidity} can be relaxed by considering IMCF in the exterior region of $M$.

\section {Preliminary}
In this section, we  introduce some notions and  present some facts from \cite{HI} that will be needed in the proof of Theorem \ref{comparisontheorem1} and Theorem \ref{rigidity}. As in \cite{HI}, a classical solution of the inverse mean curvature flow (IMCF) in $(M^3, g)$ is a smooth family of $F: N\times [0, T]\to M$  of embedded hypersurfaces $N_t= F(N,t)$ satisfying the following evolution equation
\begin{equation}\label{imcf}
\frac{\p F}{\p t}=H^{-1}\nu, \quad 0\leq t\leq T
\end{equation}
where $H$ is the mean curvature of $N_t$ at $F(x, t)$ with respect to the outward unit normal vector $\nu$ for any $x \in N$. In generally, the evolution equation (\ref{imcf}) has no classical solution. In order to overcome this difficult, a level set approach was established in \cite{HI}, i.e. these  evolving surfaces were given as the level-sets of a scalar function $u$ via $N_t=\p\{x\in M: u(x)<t\}$, where $u$ satisfies the following degenerate elliptic equation in weak sense.

\begin{equation}\label{scalarimcf}
div_M (\frac{\nabla u}{|\nabla u|})=|\nabla u|.
\end{equation}
Here the left-hand side describes the mean curvature of level-sets and the right-hand side yields the inverse speed.

By the definition of AF manifolds, for any $x\in M\setminus K$, we may consider standard coordinates $x=(x^1, x^2, x^3)$ on $\R^3$. It was observed in \cite{HI} that $v(x)=C \log |x|$ is a weak subsolution of (\ref{scalarimcf}) on $M\setminus K$ (please see the precise definition  of weak subsolution of (\ref{scalarimcf}) on P.365 in \cite{HI} ), where $|x|=\sqrt{(x^1)^2 +(x^2)^2 +(x^3)^2 }$.  With this weak subsolution one is  able to prove the existence of the weak solution of (\ref{scalarimcf}) on $M$ with any nonempty precompact smooth open set $E_0$ as initial condition (See Theorem 3.1 in \cite{HI}). Let $u^\epsilon$  be the solution of the following  elliptic regularization:
\begin{equation}\label{ellipticregularization}
\left\{
  \begin{array}{ll}
   E^\epsilon u^\epsilon= div(\frac{\nabla u^\epsilon}{\sqrt{|\nabla u^\epsilon|^2 +\epsilon^2}})-\sqrt{|\nabla u^\epsilon|^2 +\epsilon^2}=0, & \text{in $\Omega_L$} \\
   u^\epsilon =0 , & \text{on $\p E_0$} \\
    u^\epsilon =L-2, & \text{on $\p F_L$}
  \end{array}
\right.
\end{equation}

Here and in the sequel, $F_L = \{v<L\}$, for any large $L>0$, and $\Omega_L = F_L \setminus \bar E_0$, let $W^\epsilon (x,z)= u^\epsilon (x)-\epsilon z$ be a function on $\Omega_L \times \R$, then we have

$$
div(\frac{\nabla W^\epsilon}{|\nabla W^\epsilon|})=|\nabla W^\epsilon|,
$$

or equivalently, the level set $N^\epsilon_t= \{(x,z)\in \Omega_L \times \R: W^\epsilon (x,z)=t\}$ is a slice of the inverse mean curvature flow in the domain $\Omega_L \times \R$ for any $t>0$, and actually it is the classical solution to (\ref{imcf}). We know from Lemma 3.5 in \cite{HI} that (\ref{ellipticregularization}) admits a classical solution. Also, we have the following compactness lemma and its proof can be found on P.398 in \cite{HI}.

\begin{lemm}\label{compactness1}
 Let $(M^3, g)$ be an AF manifold, and $E_0$ be a precompact set of $M$ with smooth boundary. Then there are subsequences $\epsilon_i \rightarrow 0$, $L_i \rightarrow \infty$, $N^i_t=N^{\epsilon_i}_t$ such that
\begin{equation}
N^i_t \rightarrow \tilde{N_t}=N_t \times \R, \quad\text{locally in $C^1$}, \quad \text{for almost every $ t\geq 0$}
\end{equation}
where $N_t=\p E_t$ and $(E_t)_{t>0}$ is the unique weak solution of (\ref{imcf}) with $E_0$ as the initial condition.
\end{lemm}

\section{Proof of the main theorems}

In this section, we first establish some lemmas, and then prove our main results. Many arguments are from \cite{BC}.  Lemma \ref{functiontofv} below plays an important role in the proof, and meanwhile we note that  many quantities involved are not  smooth along the  weak solution of inverse mean curvature flow (\ref{imcf}). To handle this difficulty,  we first calculate the corresponding quantities along the solutions to elliptic regularizations with suitable boundary conditions. Passing to the limit using  Lemma \ref{compactness1}, we get what we want.

Let $B_\mu(x)$ be any geodesic ball with radius $\mu>0$ and center $x$  in $(M, g)$,  and let $E_0=B_\mu(x)$.  We consider the following boundary problem

\begin{equation}\label{ellipticregularization2}
\left\{
  \begin{array}{ll}
   E^\epsilon u^\epsilon= div(\frac{\nabla u^\epsilon}{\sqrt{|\nabla u^\epsilon|^2 +\epsilon^2}})-\sqrt{|\nabla u^\epsilon|^2 +\epsilon^2}=0, & \text{in $\Omega_L$} \\
   u^\epsilon =0 , & \text{on $\p E_0$} \\
    u^\epsilon =L-2, & \text{on $\p F_L$}.
  \end{array}
\right.
\end{equation}

  Using Lemma \ref{compactness1}, we know there are subsequences $\epsilon_i \rightarrow 0$, $L_i \rightarrow \infty$, $N^i_t=N^{\epsilon_i}_t$ such that
\begin{equation}
N^i_t \rightarrow \tilde{N_t}=N_t \times \R, \quad\text{locally in $C^1$}, \text{for almost every $ t\geq 0$}
\end{equation}
where $N_t=\p E_t$ and $(E_t)_{t>0}$ is the unique weak solution of (\ref{imcf})  with  the initial condition $E_0=B_\mu(x)$.  For simplicity, as in the proof of Lemma 8.1 in \cite{HI}, for each $\mu>0$, we may take a suitable transformation on $t$, so that the weak solution $(E_t)$ for the initial value problem (\ref{imcf}) is defined on $[-T(\mu), \infty)$. Here $T(\mu)\rightarrow \infty$ as $\mu$ approaches to zero, and $(E_t)_{-T(\mu)\leq t <\infty}$ converges  locally in  $C^1$ to $(G_t)_{-\infty <t <\infty}$ which is the weak solution of (\ref{imcf}) with the single point $\{x\}$  as the initial condition.

Let $W^\epsilon$ be defined by (\ref{ellipticregularization2}), and

$$
V_\epsilon(t)= Vol(\{(x,z)\in \Omega_L \times \R: W^\epsilon (x,z)<t, \quad |z|\leq \frac12)\}).
$$

Note that the level sets of  $W^\epsilon$ form a classical solution to  (\ref{imcf}). We see that $V_\epsilon(t)$ is a smooth function of $t$, and further more, we have the following result

\begin{lemm}\label{vderivativet} Let $\chi_{\{|z|\leq \frac12\}}(x,z)$ be the characteristic function of the domain $\mathbb{D}=\{(x,z)\in \Omega_L \times \R: |z|\leq \frac12\}$. Then
$$
\frac{d V_\epsilon}{dt}=\int_{N^\epsilon_t}H^{-1}_\epsilon \chi_{\{|z|\leq \frac12\}}(x,z)dS >0.
$$
Here and in the sequel, $H_\epsilon$ denotes the mean curvature of $N^\epsilon_t$ in $\mathbb{D}$ with respect to the  unit normal direction $\frac{\nabla W^\epsilon}{|\nabla W^\epsilon|}$ .
\end{lemm}
\begin{proof}
Using  the Co-area formula, we see that
\begin{equation}
\begin{split}
V_\epsilon (t)&=\int_\mathbb{D}\chi_{\{|z|\leq \frac12\}}(x,z)\chi_{\{W^\epsilon <t\}}(x,z)dv\\
&=\int^\infty_{-\infty}\int_{\{W^\epsilon =\sigma\}}\frac{\chi_{\{|z|\leq \frac12\}}(x,z)\chi_{\{W^\epsilon <t\}}(x,z)}{|\nabla W^\epsilon|}dSd\sigma\\
&=\int^t_{-\infty}\int_{\{W^\epsilon =\sigma\}}\frac{\chi_{\{|z|\leq \frac12\}}(x,z)}{|\nabla W^\epsilon|}dSd\sigma
\end{split}
\end{equation}
which implies
$$
\frac{d V_\epsilon}{dt}=\int_{N^\epsilon_t}H^{-1}_\epsilon \chi_{\{|z|\leq \frac12\}}(x,z)dS >0.
$$
This finishes the proof of the lemma.

\end{proof}

A direct conclusion of Lemma \ref{vderivativet} is the following

\begin{coro}\label{tderivativev}
Let $W^\epsilon$ be a classical solution to  (\ref{imcf}) on $\mathbb{D}$ and $v=Vol(\{(x,z)\in \Omega_L \times \R: W^\epsilon (x,z)<t, \quad |z|\leq \frac12\})$. Then $t$ is a smooth function of $v$ and
\begin{equation}
\begin{split}
\frac{dt}{dv}&=(\int_{N^\epsilon_t}H^{-1}_\epsilon \chi_{\{|z|\leq \frac12\}}(x,z)dS)^{-1}\\
&=(\int_{N^\epsilon_t \cap \{|z|\leq \frac12\}}H^{-1}_\epsilon dS)^{-1}.
\end{split}\nonumber
\end{equation}

\end{coro}

Let $(G_t)_{t>-\infty}$ be the weak solution of (\ref{imcf}). We have the following

\begin{lemm}\label{imcfvol}
For any $v>0$ either  there is a time $t \in \mathbb{R}$ with $Vol (G_t)=v$  or   $v$ is a jump  volume for (\ref{imcf}), i.e. there exists a time $t_1>-\infty$ with

$$
Vol(G_{t_1})<v\leq Vol(G^+_{t_1}),
$$
where $G^+_{t_1}$ is the strictly minimizing hull for $G_{t_1}$.
\end{lemm}
\begin{proof}
Let
$$
t_0=\inf\{t \in \R: Vol(G_t)\geq v\},
$$

and
$$
\tau_0=\sup\{t \in \R: Vol(G_t)\leq v\}.
$$

 Note that $t_0 \geq \tau_0$. By \cite{HI}, we know that $K_t=\p G_t$   converges to $K^+ _{t_0}$ locally in $C^1$ when  $t$ decreases to  $t_0$   and $K_t$   converges to $K _{\tau_0}$ locally in $C^1$  when  $t$ increases to  $\tau_0$  so that $Vol(G^+ _{t_0})\geq v \geq Vol(G _{\tau_0})$. If $t_0 > \tau_0$,  this contradicts the definition of $t_0$ or
$\tau_0$. Thus $t_0 =\tau_0$. Thus either $v$ satisfies $Vol(G_{t_0})=v$ or   $Vol(G_{t_0})<v\leq Vol(G^+_{t_0})$.  This finishes the proof of Lemma \ref{imcfvol}.

\end{proof}

The next lemma is on the relation between $t$ and the volumes of the level sets of a weak solution of (\ref{imcf}).

\begin{lemm}\label{functiontofv}
For any $v>0$, let
\begin{equation}
t(v)=\inf\{\tau: Vol(G_\tau)\geq v\}\nonumber.
\end{equation}

Then $t$ is a Lipschitz function and
\begin{equation}
\frac{dt}{dv}\leq (\int_{K_t}H^2 )^\frac12 \cdot (Area(K_t))^{-\frac32},\nonumber
\end{equation}
where $K_t=\p G_t$.
\end{lemm}

\begin{proof}
For any fixed $v>0$, let $t^i(v) =t^i$ with $v=Vol(\{(x,z)\in \Omega_L \times \R: W^\epsilon (x,z)<t^i, \quad |z|\leq \frac12\})$. Then by Lemma \ref{compactness1}, we see that $t^i(v)$ converges to $t(v)$. (Here we assume  without loss of generality the initial condition $B_\mu(x)$ shrinks to $\{x\}$ as $i \to \infty$.) Next, according to Corollary \ref{tderivativev}
\begin{equation}
\begin{split}
\frac{dt^i}{dv}&=(\int_{N^i_t \cap \{|z|\leq \frac12\}}H^{-1}_i dS)^{-1}\\
&\leq (\int_{N^i_t \cap \{|z|\leq \frac12\}}H^2_i dS)^\frac12 (Area (N^i_t \cap \{|z|\leq \frac12\}))^{-\frac32}
\end{split}\nonumber
\end{equation}
Hence, for any $v_1 \geq v_2$, we have
$$
t^i (v_1)-t^i (v_2)\leq \int^{v_1}_{v_2}(\int_{N^i_t \cap \{|z|\leq \frac12\}}H^2_i dS)^\frac12 (Area (N^i_t \cap \{|z|\leq \frac12\}))^{-\frac32} dv
$$

According to (5.6) in \cite{HI}, we see that for any $T>-T(\mu)$  and  all $t\in [-T(\mu),T]$
 $$\int_{N^i_t \cap \{|z|\leq \frac12\}}H^2_i dS\leq C(T),$$
here $C(T)$ is a constant  that depends only on $T$. Using also (5.12) in \cite{HI} we see that for almost every $t>-T(\mu)$, we have
$$
\int_{N^i_t\cap \{|z|\leq \frac12\}}H^2_i dS\rightarrow \int_{\tilde N_t\cap \{|z|\leq \frac12\}}H^2 dS.
$$

Letting $i \rightarrow \infty$ and using the bounded convergence theorem, we see that
\begin{equation}
\begin{split}
t(v_1)-t (v_2)&\leq \int^{v_1}_{v_2}(\int_{\tilde N_t\cap \{|z|\leq \frac12\}}H^2)^\frac12 (Area (\tilde N_t\cap \{|z|\leq \frac12\}))^{-\frac32} dv\\
&=\int^{v_1}_{v_2}(\int_{K_t}H^2 )^\frac12 \cdot (Area(K_t))^{-\frac32} dv.
\end{split}
\end{equation}

This finishes the proof of the lemma.

\end{proof}

Let $M_{ext}$ be the exterior region of $(M^3, g)$ defined in  Lemma 4.1 in \cite{HI}. Let $\Omega \subset M$ be a Borel set with finite perimeter, and $\Omega_{ext}= \Omega \cap M_{ext}$, here and in the sequel $M_{ext}$
is the exterior region of $M$, for the  its definition  see Lemma 4.1, P.392,  \cite{HI}. Let

$$
A_{ext}(v)= \inf \{\mathcal{H}^2 (\partial^* \Omega_{ext}): \mathcal{L}^3 (\Omega_{ext})=v \}.
$$

Clearly, we have $A(v)\leq A_{ext}(v)$. In order to prove Theorem \ref{comparisontheorem1} and Theorem \ref{rigidity}, we need $A_{ext}(v)$ to be  nondecreasing:

\begin{lemm}\label{nondecreasav}
Let $(M^3, g)$ be an AF manifold with nonnegative scalar curvature. Let $M_{ext}$ be the exterior region of $M$, then $A_{ext}(v)$ is nondecreasing.
\end{lemm}

We will use an idea from \cite{B} to prove this lemma: we need to construct a compact manifold with compact boundary from  $M_{ext}$. More precisely, note that $(M^3, g)$ is AF, hence we may take  a large compact domain $\Omega \subset M_{ext}$ so that  $M_{ext}\setminus \Omega$ is diffeomorphic to $\R^3 \setminus \mathbb{B}_{R+4}$,  hence, for simplicity, we just assume $\Omega \setminus K $ is differmorphic to $\mathbb{B}_{R+4} \setminus \mathbb{B}_\frac R2$, here $K$ is the compact domain of $M$ which appears in Definition \ref{defaf}.   On the other hand, we observe  that the standard sphere with radius $\frac\lambda 2$ can be expressed as $\mathbb{S}^2 (\lambda)=(\R^3, g_S=\frac{(dx^1)^2 + (dx^2)^2 +(dx^3)^2}{(1+\lambda^{-2}|x|^2)^2})$.  Let

\begin{equation}
\bar g= \left\{
        \begin{array}{ll}
          g, \quad \text{inside $\mathbb{B}_{R+5}$} \\
         \eta g +(1-\eta)g_S, \text{on $\mathbb{B}_{R+6} \setminus\mathbb{B}_{R+5}$} \\
          g_S, \quad \text{outside $\mathbb{B}_{R+6}$}
        \end{array}
      \right.
\end{equation}

where $\eta$ is a smooth function with $\eta=1$ in $\mathbb{B}_{R+5}$ and that vanishes outside $\mathbb{B}_{R+6}$. Thus $(M_{ext} , \bar g)$ can be regarded as a compact manifold  with compact boundary. We denote this manifold by $(\bar M, \bar g)$.

We also  need the following result from \cite[Lemma 1]{MY}.

\begin{lemm}\label{areabound}[Meeks-Yau] Let $\iota$ be the infinmum of the injectivity
radius of points in $\{x\in \bar M|\ d(x, S_\frac R2)>\frac
d4\}$. Let $K>0$ be the upper bound of the  curvature of $\bar M$ outside $\mathbb{B}_\frac R2$. Let $S_\frac R2$ be the coordinate sphere with radius $\frac R2$,  suppose $N$ is a minimal surface and suppose $x\in N$
is a point satisfying  $d(x,S_\frac R2)\ge
\frac d2$, then
\begin{equation}\label{area}
   |N\cap B_x(r)|\ge 2\pi K^{-2}\int_0^r \tau^{-1}(\sin
   K\tau)^2d\tau
\end{equation}
where $r=\min\{\frac d4,\iota\}$.
\end{lemm}

\begin{proof}[Proof of Lemma \ref{nondecreasav}]
 Assume that $A_{ext}(v)$ is not  nondecreasing. Then there are $v_1 < v_2$ with $A_{ext}(v_1)> A_{ext}(v_2)$.  Using result from geometry measure theory (\cite{Si}), there is a compact domain $\Omega_0 \subset \bar M$ with compact boundary $\Sigma_0$ and $\Sigma_0 \setminus \partial M_{ext}$ is smooth, and

$$
Area_{\bar g}(\Sigma_0)=\inf \{Area_{\bar g} (\p \Omega): \Omega \subset \bar M, Vol_{\bar g}(\Omega)\geq v_1\}.
$$
Here and in the sequel $Area_{\bar g}$, and $Vol_{\bar g}$ denote area and volume with respect to metric $\bar g$ respectively.
{\it We claim that $Vol_{\bar g}(\Omega_0)>v_1$ provided $R$ and $\lambda$ are large enough. Therefore, $\Sigma_0 \setminus \partial M_{ext}$ is a stable minimal surface in $\bar M$}. In fact, suppose $Vol_{\bar g}(\Omega_0)=v_1$,  for any $\epsilon>0$, we assume there is a compact domain $\mathbb{D}_2 \subset M$ with $Vol_g (\mathbb{D}_2 )=v_2$ and $Area_g (\p\mathbb{D}_2  )< A_{ext}(v_2)+\epsilon$, and without loss of generality, we assume $\mathbb{D}_2$ is contained in $\Omega$, then we have

\begin{equation}\label{ineq1}
Area_{\bar g}(\Sigma_0)\leq Area_{\bar g} (\p\mathbb{D}_2  )=Area_g (\p\mathbb{D}_2  )< A_{ext}(v_2)+\epsilon < A_{ext}(v_1),
\end{equation}

which implies $\Omega_0$ cannot be contained in $\Omega$ completely.

If $\Omega_0$ is contained the domain outside $\mathbb{B}_{R+6}$, then by the solution of isoperimetric problem on the standard sphere, we see that when $R$  and $\lambda$ become large enough,  the diameter of  $\Omega_0$  in $\bar M $ is uniform bounded independently of $R$ and $\lambda$. However, for any fixed $R$, taking $\lambda$ large enough, we see that the metric $\bar g$ restricted on $\Omega_0$ is almost Euclidean. Then, by a translation in $\R^3$, we may find a domain $\Omega_1$ which is contained in $\mathbb{B}_R \setminus \mathbb{B}_\frac R2 \subset \Omega $ and  isometric to $\Omega_0$ in $\R^3$. Hence, the volume of $\Omega_1$ and area of the boundary of $\Omega_1$ are very close to those of $\Omega_0$ with respect to metric $\bar g$ provided $R$ and $\lambda$ is large enough. By a small perturbation on $\Omega_1$ if necessary, we may assume $Vol_{\bar g} (\Omega_1)=Vol_{\bar g} (\Omega_0)$, and $A_{ext}(v_1) \leq Area_{\bar g}(\p \Omega_1)\leq Area_{\bar g}(\Sigma_0)+\epsilon$, which is contradiction to (\ref{ineq1}), provided that $\epsilon >0$ is sufficiently small.

For the remaining case, by the co-area formula, we see that we may find a coordinate sphere $S_\rho$ with $Area _{\bar g}(S_\rho \cap \Omega_0)<\epsilon$, and $R+6\leq \rho\leq 2R$. By the solution of  the classical isoperimetric problem on the standard sphere, we may assume the diameter of the part of $\Omega_0$ which outside $\mathbb{B}_\rho$ has uniform bounded independently of $R$ and $\rho$.  Therefore, by the same reasoning as  above, we may translate the part of $\Omega_0$ which outside $\mathbb{B}_\rho$ into $\mathbb{B}_{R}\setminus \mathbb{B}_\frac{R}{2}$ completely and get a new domain denoted by $\Omega_2$ which may have several  connected components and contained in $\mathbb{B}_{R}$. Note that $g$ is asymptotically flat. We see that the volume of $\Omega_2$ and area of the boundary of $\Omega_2$  with respect to $g$ are very close to these of $\Omega_0$. By a perturbation of $\Omega_2$ if necessary, we  get a domain in $\Omega$ which is still denoted by $\Omega_2$ with $Vol_{\bar g}(\Omega_2)=Vol_g(\Omega_2)=v_1$.  We again get $A_{ext}(v_1) \leq Area_{\bar g}(\p \Omega_2)\leq Area_{\bar g}(\Sigma_0)+2\epsilon$, which is contradiction to (\ref{ineq1}), provided $\epsilon >0$ is small enough. Therefore, $Vol_{\bar g}(\Omega_0)> v_1$, and hence, as we claimed before, $\Sigma_0\setminus \partial M_{ext}$ is a stable minimal surface in $\bar M$.

Finally, we want to prove the minimal surface $\Sigma_0$ is contained in $\mathbb{B}_{R+1}$ when $R$ is large enough. In particular, it is in $\Omega$.   Actually, for any $x\in \Sigma_0 \setminus \mathbb{B}_{R+1}$, note that  $(M,g)$ is AF. Thus, we may assume that $\iota >\frac R2$ and $K\leq  C R^{-3}$ outside $\mathbb{B}_{R+1}$. By (\ref{area}), we have that
$$
Area_{\bar g} (\Sigma_0)\geq C R^2,
$$
where $C$ is a  constant independent of $R$. However, by (\ref{ineq1}), we see that  this is a contradiction when $R$ is sufficiently  large. Thus, $\Sigma_0\setminus \partial M_{ext}$ is contained in $\Omega$. Without loss of generality, we may assume that $M_{ext}$ is foliated by spheres of positive mean curvature. It follows there are no minimal surfaces contained in $M_{ext}$. This finishes proof of the lemma.

\end{proof}

Now, we can prove Theorem \ref{comparisontheorem1} and Theorem \ref{rigidity}.
\begin{proof}[Proof of Theorem \ref{comparisontheorem1} and Theorem \ref{rigidity}] For any $v>0$, we may choose a sufficiently large $\rho=\rho(v)$, for any $x\in M\setminus B_{\rho}(0)\subset M_{ext}$ and consider the IMCF (\ref{imcf}) with initial condition $\{x\}$. By choosing $\rho >0$  sufficiently large  if necessary, we may assume there is a $G_t$ which is a domain in the weak solution of (\ref{imcf}) with initial condition $\{x\}$, satisfying $Vol(G_t)>v$, and $G_t$ is contained in the interior part of $M_{ext}$.  By a direct computation and Lemma \ref{functiontofv} , we see that
$$
\frac{dB}{dv}\leq B^{-\frac12}(\int_{K_t} H^2)^\frac12.
$$

By the definition of the Hawking mass of $K_{t(v)}$, we see that

$$
\int_{K_t} H^2= 16\pi -(16\pi)^\frac32 B^{-\frac12} m(v),
$$
 In conjunction with previous inequality, we obtain (see also Proposition 3 in \cite{BC}).

\begin{equation}\label{isoperineq1}
B(v)\leq (36\pi)^\frac13 \left(\int^v_0(1-(16\pi)^{\frac12}B^{-\frac12} (t) m(t))^\frac12  dt\right)^\frac23  .
\end{equation}

If $v$ is not a jump volume, then there is a $G_t$ with $Vol(G_t)=v$ so that in this case we have

\begin{equation}
\begin{split}
A(v)&\leq A_{ext}(v)\leq Area(K_t)=B(v)\\
&\leq (36\pi)^\frac13 \left(\int^v_0(1-(16\pi)^{\frac12}B^{-\frac12} (t) m(t))^\frac12  dt\right)^\frac23  ;
\end{split}\nonumber
\end{equation}

If $v$ is  a jump volume, then there is a $G_\tau$ with $v_1 =Vol (G_\tau)<v \leq Vol (G^+ _\tau)=v_2 $. Hence $t(v)=\tau$ and thus $B(v)=B(v_1)$,

$$
A(v)\leq A_{ext}(v)\leq A_{ext}(v_2)\leq Area (K^+ _\tau)=Area(K_\tau)=B(v_1)=B(v).
$$
Here we have used Lemma \ref{nondecreasav} in the first inequality. This finishes proof of Theorem \ref{comparisontheorem1}.

Suppose there is a $v_0>0$ with $A(v_0)=(36\pi)^\frac13 v_0 ^\frac23$. We claim that in this case $v_0$ is not a jump volume. Suppose not, then we may find $v_1$, $v_2$ with $v_1 <v_0 \leq v_2$, and $Vol(G_{t_1})=v_1$ and $Vol(G ^+ _{t_1})=v_2$. Since $A_{ext}(v)$ is nondecreasing , we see that $A_{ext}(v_1) \leq A_{ext}(v_0)\leq A_{ext}(v_2)$. However,
$$
A_{ext}(v_1)\leq Area(K_{t_1})\leq (36\pi)^\frac13 v^\frac23 _1,
$$

$$
A(v_0)=(36\pi)^\frac13 v_0 ^\frac23,
$$

$$
A_{ext}(v_2)\leq Area(K ^+ _{t_1} ) =Area (K _{t_1}).
$$

Combine these inequalities, we see that $v_0\leq v_1$, which is a contradiction. Thus  $v_0$ is not a jump volume.

Suppose there is non-flat point $x\in M\setminus B_{\rho}(0)$. We consider the weak solution of (\ref{imcf}) with initial condition $\{x\}$. By Lemma 8.1 in \cite{HI}, $m(v)>0$, for $v>0$,  together with  (\ref{isoperineq1}), we that there is $t>-\infty$ with $Vol (G_t)=v_0 $. Thus,

 $$
A(v_0)\leq B(v_0)< (36\pi)^\frac13  v_0 ^\frac23,
$$
 which is a contradiction. Thus $M\setminus B_{\rho}(0)$ is flat.  It follows that the ADM mass of $(M,g)$ is zero so that $(M,g)$ is flay by the positive mass theorem proved in \cite{SY}, we see $(M,g)=\mathbb{R}^3$.  This finishes proof of Theorem \ref{rigidity}.

\end{proof}

\end{document}